\documentclass[11pt]{article}
\usepackage{amsmath}
\usepackage{amsfonts} \makeatother
\usepackage{amsfonts,amsmath,amssymb,mathrsfs}
\usepackage{cite}
\usepackage{amsthm}

\numberwithin{equation}{section}
\newtheorem{theorem}{Theorem}[section]

\newtheorem{lemma}[theorem]{Lemma}

\newtheorem{remark}{Remark}[section]

\newcommand{\be}{\begin{equation}}
\newcommand{\ee}{\end{equation}}
\newcommand\bes{\begin{eqnarray}}
\newcommand\ees{\end{eqnarray}}
\newcommand{\bess}{\begin{eqnarray*}}
\newcommand{\eess}{\end{eqnarray*}}

\def\R{\mathbb{R}}
\def\N{\mathbb{N}}

\begin{document}

\begin{center} {\bf\Large Positive solutions for the Schr\"{o}dinger-Poisson system

   with steep potential well
 \footnote{This work was supported by the National Natural Science Foundation of China
 (No.11901284), the Natural Science Foundation of Jiangsu Province (No.BK20180814),
 the Natural science fund for colleges and universities in Jiangsu Province (No.18KJB110009)
 and Jiangsu Planned Projects for Postdoctoral Research Funds (No.2019K097).}
} \\[4mm]
 {\large Miao Du$^{\text{a}}$} \\[2mm]
{

   $^{\text{a}}$ School of Applied Mathematics, Nanjing University of Finance and Economics,

     Nanjing 210023, P.R. China
}
\renewcommand{\thefootnote}{}
\footnote{\hspace{-1ex}\emph {~E-mail addresses:}
dumiaomath@163.com.}
\end{center}

\setlength{\baselineskip}{16pt}

\begin{quote}
  \noindent {\bf Abstract:} In this paper, we consider the following
  Schr\"odinger-Poisson system
  \vskip -0.3 true cm
  \begin{equation*}
   \begin{cases}
     -\Delta u+\lambda V(x)u+ \mu\phi u=|u|^{p-2}u &\text{in $\R^{3}$},
       \vspace{0.1cm}\\
     -\Delta \phi=u^{2} &\text{in $\R^{3}$},
   \end{cases}
  \end{equation*}
  \vskip -0.1 true cm
  where $\lambda,\:\mu>0$ are real parameters and $2<p<6$. Suppose that
  $V(x)$ represents a potential well with the bottom $V^{-1}(0)$,
  the system has been widely studied in the case $4\leq p<6$.
  In contrast, no existence result of solutions is available for
  the case $2<p<4$ due to the presence of the nonlocal term $\phi u$.
  With the aid of the truncation technique and the parameter-dependent
  compactness lemma, we first prove the existence of positive solutions
  for $\lambda$ large and $\mu$ small in the case $2<p<4$.
  Then we obtain the nonexistence of nontrivial solutions
  for $\lambda$ large and $\mu$ large in the case $2<p\leq3$.
  Finally, we explore the decay rate of the positive solutions
  as $|x| \rightarrow \infty$ as well as their asymptotic behavior
  as $\lambda \rightarrow \infty$ and $\mu \rightarrow 0$.

  \vskip 0.05 true cm
  \noindent {\bf Keywords}: {Schr\"odinger-Poisson system; Positive solution;
    Asymptotic behavior; Truncation technique; Variational method}

  \vskip 0.05 true cm
  \noindent {\bf MSC 2010}: 35J20; 35J60; 35B40
\end{quote}

\section{Introduction}

\indent

The present paper is devoted to investigate the existence and asymptotic behavior of
positive solutions for the following Schr\"odinger-Poisson system
\vskip -0.2 true cm
\begin{equation}\label{eq 1.1}
 \begin{cases}
  -\Delta u+\lambda V(x)u+ \mu\phi u=|u|^{p-2}u &\text{in $\R^{3}$},
       \vspace{0.1cm}\\
  -\Delta \phi=u^{2} &\text{in $\R^{3}$},
 \end{cases}
\end{equation}
where $\lambda,\:\mu>0$ are real parameters, $2<p<6$ and the potential
$V$ satisfies the following conditions:

\vspace{0.1 cm}
\noindent \hspace{0.2 cm}
$(V_{1})$  $V\in C(\mathbb{R}^{3},\:\mathbb{R})$ and $V\geq 0$
           on $\mathbb{R}^{3}$.

\noindent \hspace{0.2 cm}
$(V_{2})$  There exists $b>0$ such that $\mathcal{V}_b:=\{x \in \R^3:
           \: V(x)<b\}$ is nonempty and has finite measure.

\noindent \hspace{0.2 cm}
$(V_{3})$  $\Omega =int\:V^{-1}(0)$ is a nonempty open set with
           locally Lipschitz boundary and $\overline{\Omega }= V^{-1}(0)$.
\vspace{0.1 cm}

\noindent
This kind of hypotheses was first introduced by Bartsch and Wang
\cite{Bartsch1995} in the study of Schr\"{o}dinger equations,
and has attracted the attention of many domestic scholars, see e.g.
\cite{Bartsch2001,Dingyanheng2007,Jiangyongsheng2011,Wangzhengping2009,Zhaoleiga2013}.
Note that, the assumptions $(V_{1})$--$(V_{3})$ imply that $\lambda V$
represents a potential well with the bottom $V^{-1}(0)$ and its steepness
is controlled by the parameter $\lambda$. As a result, $\lambda V$
is often known as the steep potential well if $\lambda$ is sufficiently large,
and we expect to find solutions which are localized near the bottom of the potential $V$.
The second equation in \eqref{eq 1.1} determines $\phi: \mathbb{R}^{3}\rightarrow\mathbb{R}$
only up to harmonic functions. It is natural to choose $\phi$ as the Newton potential of $u^{2}$,
that is, the convolution of $u^{2}$ with the fundamental solution $\Phi$ of the Laplacian,
which is given by $\Phi(x)= (4 \pi |x|)^{-1}$. Denote by $\phi_u = \Phi \ast |u|^{2}$,
then with this formal inversion of the second equation in \eqref{eq 1.1}, we obtain
the integro-differential equation
\begin{equation}\label{eq 1.2}
   -\Delta u+ \lambda V(x)u + \mu \phi_u u =|u|^{p-2}u
   \quad \text{in}\ \mathbb{R}^{3}.
\end{equation}
Obviously, $(u, \phi_u)$ is a solution of \eqref{eq 1.1} if and only if $u$ is a solution of \eqref{eq 1.2}.

System \eqref{eq 1.1}, also known as Schr\"{o}dinger-Maxwell system, arises
in many problems of physics. 
We refer the reader e.g. to \cite{Mauser-2001}, where \eqref{eq 1.2} is
discussed in a quantum mechanical context where the particular exponent
$p=8/3$ appears in this case, see \cite[p. 761]{Mauser-2001}.
The unknowns $u$ and $\phi$ represent the wave functions associated with
the particle and electric potential, and the function $V$ is an external potential.
We refer to Benci and Fortunato \cite{Benci1998} for more details.
This model can also appear in semiconductor theory to describe solitary waves
\cite{Markowich1990,Sanchez2004}.
In recent years, the following Schr\"odinger-Poisson system
\begin{equation*}
 \begin{cases}
  -\Delta u+ V(x)u+K(x)\phi u=f(x,u) &\text{in} \  \mathbb{R}^{3}\\
  -\Delta \phi=K(x)u^{2}             &\text{in} \  \mathbb{R}^{3}
 \end{cases}
\end{equation*}
have been widely investigated,
whereas existence, nonexistence and multiplicity results have been obtained
under variant assumptions on $V$, $K$ and $f$ via variational methods,
see e.g. \cite{Ambrosetti-Ruiz-2008,Azzollini-Pomponio-2008,Bellazzini-Jeanjean-Luo-2008,
Cerami-2010,Chen-Tang-2020,DAprile-2004-1,DAprile-2005,Ruiz-2006,Sunjuntao2020,Zhaoleiga-2008}
and the references therein.

Inspired by \cite{Dingyanheng2007,Jiangyongsheng2011}, Zhao et al. \cite{Zhaoleiga2013} studied the system
\begin{equation}\label{eq 1.3}
 \begin{cases}
  -\Delta u+\lambda V(x)u+ K(x)\phi u =|u|^{p-2}u  &\text{in}\ \mathbb{R}^{3},\\
  -\Delta \phi=K(x)u^{2} &\text{in}\ \mathbb{R}^{3}.
 \end{cases}
\end{equation}
In this work, a positive function $K \in L^{2}(\R^{3})\cup L^{\infty}(\R^{3})$
and $3<p<6$ are considered. By using variational methods, the existence and
asymptotic behavior of nontrivial solutions were detected in \cite{Zhaoleiga2013}.
In particular, the potential $V$ is allowed to be sign-changing for the case $4<p<6$.
We would like to point out that the parameter-dependent compactness lemma
(see \cite[Lemma 2.6]{Zhaoleiga2013}) for the case $3<p<4$ relies heavily
on the condition $K \in L^2(\R^3)$ (this condition can weaken the strong influence of the nonlocal term),
and hence the authors dealt only with the case $3<p<4$ for $K \in L^2(\R^3)$. In addition, the approach
(Jeanjean's monotonicity trick \cite{Jeanjean-1999}) used in \cite{Zhaoleiga2013}
does not work any more for the case $2<p\leq3$, since we might not obtain a bounded Palais-Smale sequence.
As far as we know, there is no existence result of solutions for \eqref{eq 1.3}
in the case where $2<p<4$ and $K \in L^{\infty}(\R^{3})$.
This gap of information is unpleasant not only from a mathematical point of view but also since,
as already remarked above, the case $p=8/3$ is relevant in 3-dimensional quantum mechanical models,
see \cite[p. 761]{Mauser-2001}.
The key difficulty in this case is the competing nature of the local and nonlocal
superquadratic terms in the functional corresponding to \eqref{eq 1.3}. In particular,
we note that the nonlinearity $u \mapsto f(u):=|u|^{p-2}u$ with $2<p<4$
does not satisfy the Ambrosetti-Rabinowitz type condition
$$
  0<\mu\int_{0}^{u} f(s)\:ds \leq f(u)u \qquad
  \text{for all $u\neq0$ with some $\mu>4$}
$$
which would readily obtain a bounded Palais-Smale sequence or Cerami sequence.
Moreover, the fact that the function $f(s)/|s|^{3}$ is not increasing
on $(-\infty,\, 0)$ and $(0,\,\infty)$ prevents us from using
Nehari manifold and fibering methods as e.g. in \cite{Rabinowitz1992,Szulkin2010}.

Motivated by the works mentioned above, the purpose of the present paper
is to consider the Schr\"odinger-Poisson system \eqref{eq 1.3}
in the case where $2<p<4$ and $K \in L^{\infty}(\R^{3})$.
More precisely, we shall first prove the existence of positive solutions
for \eqref{eq 1.3} for $\lambda$ large and $\mu$ small in this case.
Then we obtain the nonexistence of nontrivial solutions for $\lambda$ large
and $\mu$ large in the case $2<p\leq3$. Finally, we explore the decay rate of
the positive solutions as $|x| \rightarrow \infty$ as well as their
asymptotic behavior as $\lambda \rightarrow \infty$ and $\mu \rightarrow 0$.
For the sake of simplicity, in the sequel we always assume that $K$ is a positive constant.
Consequently, we are dealing with the system \eqref{eq 1.1}, the associated scalar equation
\eqref{eq 1.2} and the associated energy functional
\begin{equation*}
  I_{\lambda,\mu} (u) = \frac{1}{2}\int_{\mathbb{R}^{3}} \left(|\nabla u|^{2}+\lambda
   V(x)u^{2}\right)dx+\frac{\mu}{4}\int_{\mathbb{R}^{3}}\phi_uu^{2}\:dx
   - \frac{1}{p}\int_{\mathbb{R}^{3}}|u^+|^p\: dx
\end{equation*}
defined in the space
$$
  E_\lambda=\Bigl\{u\in H^{1}(\mathbb{R}^{3}):\:
  \int_{\mathbb{R}^{3}}V(x)u^{2}\: dx < \infty\Bigr\}
$$
endowed with the norm
\begin{equation*}
   \|u\|_\lambda=\Bigl(\int_{\mathbb{R}^{3}}\left( |\nabla u|^2 +
   \lambda V(x)u^2\right)dx\Bigr)^{1/2},
\end{equation*}
where $u^+= \max\{u, \,0\}$. Our first main result is concerned with the existence of positive solutions.



\begin{theorem}\label{th 1.1}
Suppose that $2<p<4$ and $(V_{1})$--$(V_{3})$ hold. Then
there exist $\lambda^\ast>1$ and $\mu_\ast>0$ such that
for each $\lambda \in (\lambda^\ast,\,\infty)$ and $\mu \in (0,\,\mu_\ast)$,
\eqref{eq 1.2} has at least a positive solution $u_{\lambda,\mu} \in E_\lambda$.
Moreover, there exist constants $\tau,\, T>0$
$($independent of $\,\lambda$ and $\mu$$)$ such that
\vskip -0.4 true cm
\begin{equation}\label{eq 1.4}
  \tau \leq \|u_{\lambda,\mu}\|_\lambda \leq T \qquad
  \text{for all $\lambda$ and $\mu$}.
\end{equation}
\end{theorem}

\begin{remark}
\indent \rm
\begin{itemize}
  \item [\rm(i)] We note that, if $u\in H_0^1(\Omega)$ is a nontrivial
   solution of the following equation
   \vskip -0.5 true cm
   \begin{equation*}
     - \Delta u + \mu \phi_u u =|u|^{p-2}u \quad \text{in  $\Omega$,}
   \end{equation*}
   \vskip -0.2 true cm
   then by zero continuation, $u$ is also a nontrivial solution of \eqref{eq 1.2}
    for all $\lambda>0$. So, we are interested in seeking for the positive solution
    of \eqref{eq 1.2}, and obviously it does not lie in $H_0^1(\Omega)$.

  \item [\rm(ii)]   Theorem \ref{th 1.1} seems to be the first existence result of solutions for
    \eqref{eq 1.2} in the case where $2< p <4$, and it implies that \eqref{eq 1.2} has at least
    a positive solution for $\lambda$ large and $\mu$ small.
\end{itemize}
\end{remark}

It is also interesting to know whether \eqref{eq 1.2} has a nontrivial solution for $\lambda$ large and $\mu$ large.
The following theorem gives an explicit answer for the case $2<p\leq3$.

\begin{theorem}\label{th 1.2}
Suppose that $(V_{1})$--$(V_{3})$ hold.
\begin{enumerate}
  \item [\rm(i)] If $2<p<3$ and $\left|\mathcal{V}_b\right| < S^{\frac{3}{2}}$,
  \eqref{eq 1.2} has no nontrivial solution in $E_\lambda$ for all $\lambda \geq 1/b$ and
  $\mu \geq 1 / \bigl[4(1-\left|\mathcal{V}_b\right|^{\frac{2}{3}}S^{-1})\bigr]$.
  Here $S$ is the best constant for the embedding $D^{1,2}(\R^3) \hookrightarrow L^6(\R^3)$.

  \item [\rm(ii)] If $ p=3 $, \eqref{eq 1.2} has no nontrivial solution in $E_\lambda$ for all
  $\lambda >0$ and $\mu > 1/4$.
\end{enumerate}
\end{theorem}

\begin{remark}
It is still an open question whether \eqref{eq 1.2}  has a nontrivial solution for $\lambda$ large and $\mu$ large
in the case $3< p <4$, which is under consideration in my following work.
\end{remark}

We are now in a position to give the main idea of the proof of Theorem \ref{th 1.1}.
If we apply the Mountain Pass Theorem directly to the functional $I_{\lambda,\mu}$,
we may then obtain a Cerami sequence for $\mu>0$ sufficiently small. However,
the boundedness of this Cerami sequence becomes a major difficulty as noted before.
To get around this obstacle, we shall use the truncation technique as e.g. in \cite{Liy2012}.
More precisely, for each $T>0$  we move to study the truncated functional
$I^{T}_{\lambda,\mu}: E_\lambda \rightarrow \R$ defined by
\begin{equation*}
  I^T_{\lambda,\mu} (u) = \frac{1}{2}\|u\|^2_{\lambda} + \frac{\mu}{4}\eta\left(\|u\|^2_{\lambda}/T^2\right)
   \int_{\mathbb{R}^{3}}\phi_uu^{2}\:dx - \frac{1}{p}\int_{\mathbb{R}^{3}}|u^+|^p\: dx,
\end{equation*}
where $\eta$ is a smooth cut-off function such that
\begin{align*}
  \eta\left(\|u\|^2_{\lambda}/T^2\right)=
  \left\{
    \begin{array}{ll}
      1  \quad & \hbox{if\, $\|u\|_{\lambda}\leq T$,} \vspace{0.15cm}\\
      0  & \hbox{if\,  $\|u\|_{\lambda}\geq \sqrt{2} T$.}
    \end{array}
  \right.
\end{align*}
At this point, we wish to outline the proof of Theorem \ref{th 1.1}.
First we show that the truncated functional $I^T_{\lambda,\mu}$
has the mountain pass geometry for $\mu>0$ sufficiently small,
and thus obtain a Cerami sequence $\{u_n\}$ of $I^T_{\lambda,\mu}$
at the mountain pass level $c^T_{\lambda,\mu}$. We then give a key observation that
$c^T_{\lambda,\mu}$ has an upper bound independent of $T, \,\lambda$ and $\mu$.
From this observation, we may follow the standard truncation argument
to deduce that for a given $T>0$ properly, after passing to a subsequence,
$\|u_n\|_{\lambda}\leq T$ for all $n \in \N$ by restricting $\mu>0$ sufficiently small again,
and so $\{u_n\}$ is a bounded Cerami sequence of $I_{\lambda,\mu}$, i.e.,
\begin{equation*}
  \sup_{n \in \N}\|u_n\|_{\lambda}\leq T, \qquad I_{\lambda,\mu}(u_{n})\rightarrow c^T_{\lambda,\mu}
   \qquad \text{and} \qquad
  (1+\|u_n\|_\lambda)\left\|I'_{\lambda,\mu}(u_{n})\right\|_{E'_\lambda}\rightarrow 0,
\end{equation*}
where $E'_\lambda$ is the dual space of $E_\lambda$.
Finally, by using the parameter-dependent compactness lemma, for $\lambda>0$ sufficiently large
we may pass to a subsequence of $\{u_n\}$ which converges to $u_{\lambda,\mu}$ in $E_\lambda$.
Therefore, $u_{\lambda,\mu}$ is a positive solution of \eqref{eq 1.2} with
$\|u_{\lambda,\mu}\|_{\lambda}\leq T$ and $I_{\lambda,\mu}(u_{\lambda,\mu})= c^T_{\lambda,\mu}$.

Next, we would like to explore the decay of the positive solutions at infinity.
Since it is possible that $\liminf_{|x|\rightarrow \infty} V(x)=0$ in our
setting, we need to replace $(V_2)$ by the following condition:
\vspace{0.2cm}

\noindent \hspace{0.2 cm}
$(V'_{2})$  There exists $b > 0$ such that $\mathcal{V}_b:=\{x \in \R^3:
           \: V(x)<b\}$ is nonempty and bounded.
\vspace{0.2cm}

\noindent
It is easy to see that $(V'_{2})$ is stronger than $(V_{2})$. Thus, under the
assumptions of Theorem \ref{th 1.1} with $(V_2)$ replaced by $(V'_{2})$,
the conclusions of Theorem \ref{th 1.1} still hold.
There are indeed many functions
satisfying $(V_{1})$, $(V'_{2})$ and $(V_{3})$. Here we give two examples.
one example is a bounded potential function:
\begin{equation*}
  V(x)=\left\{
    \begin{array}{ll}
      \: 0                  & \hbox{if\, $|x| \leq 1$,}\vspace{0.15cm}\\
      \bigl(|x|-1\bigr)^2    \quad  & \hbox{if\, $1< |x| \leq2$,} \vspace{0.15cm}\\
      \:1                  & \hbox{if\, $|x| > 2$.}
    \end{array}
  \right.
\end{equation*}
Another example is a unbounded potential function:
\begin{equation*}
  V(x)=\left\{
    \begin{array}{ll}
      \: 0                  & \hbox{if\, $|x| \leq 1$,}\vspace{0.15cm}\\
      \bigl(|x|-1\bigr)^2    \quad  & \hbox{if\, $|x|> 1$.}
    \end{array}
  \right.
\end{equation*}

Now we are ready to investigate the decay rate of the positive solutions at infinity.
The following result shows that the positive solutions of \eqref{eq 1.2}
decay exponentially as $|x|\rightarrow \infty$.

\begin{theorem}\label{th 1.3}
Suppose that $2<p<4$, $(V_{1})$, $(V'_{2})$ and $(V_{3})$ hold.
Let $u_{\lambda,\mu}$ be the positive solution of equation \eqref{eq 1.2} satisfying \eqref{eq 1.4}
for each $\lambda \in (\lambda^\ast,\,\infty)$ and $\mu \in (0,\,\mu_\ast)$.
Then there exists $\Lambda^{\ast}>\lambda^\ast $ such that
for each $\lambda \in (\Lambda^{\ast},\,\infty)$ and $\mu \in (0,\,\mu_\ast)$, we have
\begin{equation*}
  u_{\lambda,\mu} (x) \leq  A \lambda^{- 1/2} \exp \left(-\beta \lambda^{1/2}
   (|x|-R)\right)  \qquad \text{for all\: $|x| > R$}
\end{equation*}
with constants $A,\, \beta, \, R>0$ independent of $\lambda$ and $\mu$.
\end{theorem}

\begin{remark}
\rm
The similar work on Schr\"{o}dinger equations can be found in \cite[Theorem 1.3]{Bartsch2001}.
We wish to point out that although the idea of the proof of Theorem \ref{th 1.3}
is inspired by \cite{Bartsch2001},  the adaptation procedure to our problem
is not trivial at all due to the the presence of the nonlocal term.
\end{remark}

Finally, we study the asymptotic behavior of the positive solutions as $\lambda\rightarrow \infty$
and $\mu \rightarrow 0$. By means of \eqref{eq 1.4}, we have the following results.

\begin{theorem}\label{th 1.4}
Let $u_{\lambda,\mu}$ be the positive solutions of \eqref{eq 1.2} obtained by Theorem \ref{th 1.1}.
Then for each $\mu \in (0,\, \mu_\ast)$ fixed,
$u_{\lambda,\mu}\rightarrow u_{\mu}$ in $H^{1}(\mathbb{R}^{3})$
as $\lambda \rightarrow \infty$ up to a subsequence,
where $u_{\mu} \in H_{0}^{1}(\Omega)$ is a positive solution of
\begin{equation*}
\begin{cases}
   - \Delta u + \mu \phi_u u  =|u|^{p-2}u  \ &\text{in $\Omega$}, \vspace{0.05cm}\\
   u =0   & \text{on  $\partial\Omega$}.
\end{cases}\eqno{(\mathcal{P}_{\infty,\mu})}
\end{equation*}
\end{theorem}

\begin{theorem}\label{th 1.5}
Let $u_{\lambda,\mu}$ be the positive solutions of \eqref{eq 1.2}
obtained by Theorem \ref{th 1.1}. Then for each $\lambda \in (\lambda^\ast,\,\infty)$ fixed,
$u_{\lambda,\mu}\rightarrow u_{\lambda}$ in $E_\lambda$
as $\mu \rightarrow 0$ up to a subsequence, where $u_{\lambda} \in E_\lambda$
is a positive solution of
\begin{equation*}
  \begin{cases}
    - \Delta u +\lambda V(x)u = |u|^{p-2}u \ \quad  \text{in $\mathbb{R}^{3}$},\vspace{0.1cm}\\
    u \in H^{1}(\mathbb{R}^{3}).
 \end{cases}\eqno{(\mathcal{P}_{\lambda,0})}
\end{equation*}
\end{theorem}

\begin{theorem}\label{th 1.6}
Let $u_{\lambda,\mu}$ be the positive solutions of \eqref{eq 1.2}
obtained by Theorem \ref{th 1.1}. Then $u_{\lambda,\mu}\rightarrow u_{0}$ in $H^{1}(\mathbb{R}^{3})$
as $\mu \rightarrow 0$ and $\lambda \rightarrow \infty$ up to a subsequence,
where $u_{0} \in H_{0}^{1}(\Omega)$ is a positive solution of
\begin{equation*}
   \begin{cases}
      - \Delta u =|u|^{p-2}u  \ &\text{in $\Omega$}, \vspace{0.05cm}\\
      u =0   & \text{on  $\partial\Omega$}.
   \end{cases}\eqno{(\mathcal{P}_{\infty,0})}
\end{equation*}
\end{theorem}

\begin{remark}
\indent \rm
\begin{itemize}
  \item [\rm(i)]   Let $\mu>0$ be a small fixed-parameter, Theorem \ref{th 1.4} shows that
     the positive solutions $u_{\lambda,\mu}$ are well localized near the bottom
     of the potential as $\lambda \rightarrow \infty$.
  \vspace{-0.15cm}
  \item [\rm(ii)]  Let $\lambda>0$ be a large fixed-parameter, Theorem \ref{th 1.5} shows that the positive
    solutions of \eqref{eq 1.2} may converge in $E_\lambda$ to a positive solution
    of $(\mathcal{P}_{\lambda,0})$ as $\mu \rightarrow 0$ up to a subsequence.
  \vspace{-0.15cm}
  \item [\rm(iii)] Theorem \ref{th 1.6} shows that the positive solutions of \eqref{eq 1.2}
     may converge in $H^{1}(\mathbb{R}^{3})$ to a positive solution of $(\mathcal{P}_{\infty,0})$
     as $\lambda\rightarrow\infty$ and $\mu \rightarrow 0$ up to a subsequence.
\end{itemize}
\end{remark}

The remainder of this paper is organized as follows.
In Section \ref{section 2}, we set up the  variational framework of
\eqref{eq 1.2} and present some preliminary results. In Section \ref{section 3},
we give the proofs of Theorems \ref{th 1.1} and \ref{th 1.2}.
Section \ref{section 4} is devoted to the proof of Theorem \ref{th 1.3}.
Finally, in Section \ref{section 5} we will complete the proofs of
 Theorems \ref{th 1.4}, \ref{th 1.5} and \ref{th 1.6}.

Throughout the paper, we make use of the following notations.
$H^{1}(\mathbb{R}^{3})$ is the usual Sobolev space endowed with
the standard scalar product and norm. $L^{s}(\mathbb{R}^{3})$,
$1\leq s \leq \infty$, denotes the usual Lebesgue space
with the norm $|\cdot|_{s}$. For any $\rho>0$ and $z\in\mathbb{R}^{3}$,
$B_{\rho}(z)$ denotes the ball of radius $\rho$ centered at $z$.
 $|M|$ is the Lebesgue measure of the set $M$.
As usual, $X'$ denotes the dual space of $X$. Finally,  $C,\: C_{1},\:C_{2},\:\cdots$
denote different positive constants whose exact value is inessential.

\section{Preliminaries}\label{section 2}

\indent

In this section, we establish the variational framework for equation \eqref{eq 1.2}
as elaborated by \cite{Dingyanheng2007} and give some useful preliminary results.
Let
\begin{equation*}
   E=\left\{u\in H^{1}(\mathbb{R}^{3}) :\: \int_{\mathbb{R}^{3}}V(x)u^{2}\:dx<\infty\right\}
\end{equation*}
be equipped with the inner product and norm
\begin{equation*}
   \langle u,\:v \rangle=\int_{\mathbb{R}^{3}}\left(\nabla u \nabla v+V(x)uv\right)dx,
   \qquad \|u\|=\langle u, \,u \rangle^{1/2}.
\end{equation*}
For $\lambda > 0$, we also need the following inner product and norm
\begin{equation*}
   \langle u,\:v \rangle_{\lambda}=\int_{\mathbb{R}^{3}}\left( \nabla u \nabla v +
   \lambda V(x)uv\right)dx, \qquad \|u\|_{\lambda}=\langle u, \,u \rangle^{1/2}_{\lambda}.
\end{equation*}
It is clear that $\|u\| \leq \|u\|_{\lambda}$ for $\lambda \geq 1$.
Set $E_{\lambda}=\left(E,\, \|\cdot\|_{\lambda}\right)$. It then follows from the conditions
$(V_{1})$--$(V_{2})$ and the H\"{o}lder and the Sobolev inequalities that
\begin{align*}
  \int_{\mathbb{R}^{3}}\left(|\nabla u|^{2}+u^{2}\right)dx
    &  = \int_{\mathbb{R}^{3}}
    |\nabla u|^{2}\:dx + \int_{\mathcal{V}_b}u^{2}\:dx + \int_{\R^3 \setminus \mathcal{V}_b}u^{2}\:dx\\
  & \leq \int_{\mathbb{R}^{3}}|\nabla u| ^{2}\:dx+\left|\mathcal{V}_b\right|^{\frac{2}{3}}
    \left(\int_{\mathcal{V}_b}|u|^{6}\:dx \right)^{\frac{1}{3}}
      + b^{-1}\int_{\R^3 \setminus \mathcal{V}_b} V(x) u^{2}\:dx\\
  &\leq \max\left\{1+\left|\mathcal{V}_b\right|^{\frac{2}{3}}S^{-1},\: b^{-1}\right\}
    \int_{\mathbb{R}^{3}}\left( |\nabla u| ^{2}+V(x)u^{2}\right)dx,
\end{align*}
which implies that the embedding $E\hookrightarrow H^{1}(\mathbb{R}^{3})$ is continuous.
Thus, for each $s \in [2,6]$, there exists $d_{s}>0$
(independent of $\lambda \geq 1$) such that
\begin{equation}\label{eq 2.1}
   |u|_{s}\leq d_{s} \|u\| \leq d_{s} \|u\|_{\lambda} \qquad \text{for $u \in E$}.
\end{equation}
From the Hardy-Littlewood-Sobolev inequality \cite{Lieb-1983},
we deduce that
\begin{equation}\label{eq 2.2}
   \int_{\mathbb{R}^{3}}\phi_u u^2\:dx = \int_{\mathbb{R}^{3}}\int_{\mathbb{R}^{3}}
   \frac{u^2(x)u^2(y)}{4 \pi |x-y|}\:dxdy \leq C_{0}|u|^4_{12/5}  \quad
    \text{for}\ u \in L^{12/5}(\mathbb{R}^{3})
\end{equation}
with a constant $C_{0}>0$. Consequently, the functional $I_{\lambda,\mu}:\: E_\lambda \rightarrow \R$
given by
\begin{equation*}
  I_{\lambda,\mu} (u) = \frac{1}{2}\int_{\mathbb{R}^{3}} \left(|\nabla u|^{2}+\lambda
   V(x)u^{2}\right)dx+\frac{\mu}{4}\int_{\mathbb{R}^{3}} \phi_u u^2 \:dx
   - \frac{1}{p}\int_{\mathbb{R}^{3}}|u^+|^p\:dx
\end{equation*}
is well defined, and it is of class $C^{1}$ with derivative
\begin{equation*}
  \bigl\langle I'_{\lambda,\mu}(u), \:v \bigr\rangle = \int_{\mathbb{R}^{3}}\left(\nabla u \nabla v + \lambda V(x)uv\right)\:dx
   + \mu \int_{\mathbb{R}^{3}} \phi_u uv\:dx - \int_{\mathbb{R}^{3}}|u^+|^{p-2}u^+ v\:dx
\end{equation*}
for all $u,\, v \in E_\lambda$. Moreover, it is well known that every nontrivial critical point
of $I_{\lambda,\mu}$ is a positive solution of \eqref{eq 1.2},  but we repeat it
for the convenience of the reader in the following lemma.

\begin{lemma}\label{lem 2.1}
Suppose that $2< p < 4$ and $(V_{1})$--$(V_{2})$ are satisfied. Then every nontrivial critical
point of $I_{\lambda,\mu}$ is a positive solution of \eqref{eq 1.2}.
\end{lemma}
\begin{proof}
Let $u \in E_\lambda$ is a nontrivial critical point of $I_{\lambda,\mu}$, then
\begin{equation}\label{eq 2.3}
   \int_{\R^3} (\nabla u \nabla v +\lambda V(x)uv) \:dx +
   \mu \int_{\mathbb{R}^{3}}\phi_u u v\:dx
   - \int_{\R^3}|u^+|^{p-2}u^+ v\:dx = 0 \quad \text{for all $v \in E_\lambda$.}
\end{equation}
 Taking $v = u^-= -\min\{u,\:0\}$ in \eqref{eq 2.3}, we obtain
$\|u^-\|^2_\lambda = 0$, and so $u\geq 0$ in $\R^3$.
Hence, the strong maximum principle and the fact $u \neq 0$ imply that $u > 0$ in $\R^3$,
and the claim follows.
\end{proof}

We close this section with a useful theorem. It is a somewhat stronger version of the
Mountain Pass Theorem, which allows us to find so-called Cerami sequences
instead of Palais-Smale sequences.

\begin{theorem}[See \cite{Ekeland1990}]\label{th 2.2}
Let $X$ be a real Banach space with its dual space $X'$, and suppose that
$J \in C^1(X,\,\R)$ satisfies
$$
  \max\left\{J(0),\, J(e)\right\} \leq \xi <\eta \leq \inf_{\|u\|_X=\rho} J(u)
$$
for some $\xi <\eta$, $\rho>0$ and $e \in X$ with $\|e\|_X>\rho$. Let $c \geq \eta$
be characterized by
$$
  c = \inf_{\gamma \in \Gamma} \max_{t \in [0,1]} J(\gamma(t)),
$$
where $\Gamma = \left\{\gamma \in C\bigl([0,1],\, X\bigr):\: \gamma(0)=0,\:\gamma(1)=e\right\}$
is the set of continuous paths joining $0$ and $e$. Then there exists a sequence $\{u_n\}\subset X$
such that
$$
  J(u_n) \rightarrow c\geq \eta \qquad \text{and} \qquad
  (1+\|u_n\|_X)\|J'(u_{n})\|_{X'}\rightarrow 0
  \quad \text{as} \ n \rightarrow \infty.
$$
\end{theorem}

\vspace{0.2cm}

\section{Existence and nonexistence of solutions to \eqref{eq 1.2}}\label{section 3}

\indent

In this section, we study the existence and nonexistence of solutions for \eqref{eq 1.2}
and give the proofs of Theorems \ref{th 1.1} and \ref{th 1.2}. For this,
we first define a cut-off function $\eta \in C^1
\left([0, \infty), \,\R\right)$ satisfying $0 \leq \eta \leq 1$,  $\eta(t)=1$ if $0 \leq t \leq 1$,
$\eta(t)=0$ if $t \geq 2$, $\max_{t>0} |\eta'(t)| \leq 2$ and $\eta'(t) \leq 0$ for each $t>0$.
Using $\eta$, for every $T>0$  we then consider the truncated functional $I^T_{\lambda,\mu}:\:
E_{\lambda} \rightarrow \mathbb{R}$ defined by
\begin{equation*}
  I^T_{\lambda,\mu} (u) = \frac{1}{2}\|u\|_\lambda^2+\frac{\mu}{4}\eta\left(\|u\|^2_{\lambda}/T^2\right)
  \int_{\mathbb{R}^{3}}\phi_u u^2\:dx - \frac{1}{p}|u^+|_p^p.
\end{equation*}
It is easy to see that $I^T_{\lambda,\mu}$ is of class $C^1$.
Moreover, for each $u,\, v \in E_\lambda$  we have
\begin{align}\label{eq 3.1}
  \bigl\langle (I^T_{\lambda,\mu})'(u),\:v\bigr\rangle&=\langle u,\:v\rangle_\lambda
   + \mu \eta\left(\|u\|^2_{\lambda}/T^2\right)\int_{\mathbb{R}^{3}}\phi_u uv\:dx\notag\\
  &\quad + \frac{\mu}{2 T^2}\eta'\left(\|u\|^2_{\lambda}/T^2\right)\langle u,\:v\rangle_\lambda
  \int_{\mathbb{R}^{3}}\phi_u u^2\:dx
  -\int_{\mathbb{R}^{3}}|u^+|^{p-2}u^+v\:dx.
\end{align}
With this penalization, by choosing an appropriate $T>0$ and restricting $\mu>0$ sufficiently small,
we may obtain a Cerami sequence $\{u_n\}$ of $I^T_{\lambda,\mu}$ satisfying $\|u_n\|_\lambda\leq T$, and so
$\{u_n\}$ is also a Cerami sequence $\{u_n\}$ of $I_{\lambda,\mu}$ satisfying $\|u_n\|_\lambda\leq T$.

To begin with, we show that the truncated functional $I^T_{\lambda,\mu}$ has
the mountain pass geometry.

\begin{lemma}\label{lem 3.1}
Suppose that $2< p < 4$ and $(V_{1})$--$(V_{2})$ hold. Then
for each $T,\:\mu>0$ and $\lambda \geq 1$, there exist $\alpha,\,\rho>0$
$($independent of $\,T,\,\lambda$ and $\mu$$)$ such that
$I^T_{\lambda,\mu}(u)\geq \alpha$ for all $u \in E_\lambda$
with $\|u\|_\lambda=\rho$.
\end{lemma}

\begin{proof}
For each $u \in E_\lambda$, by \eqref{eq 2.1} we have
$$
  I^T_{\lambda,\mu}(u) \geq \frac{1}{2}\|u\|^2_\lambda - \frac{1}{p} d_p^p\|u\|^p_\lambda
   = \|u\|^2_\lambda\left(\frac{1}{2} - \frac{1}{p}d_p^p \|u\|^{p-2}_\lambda\right) ,
$$
where the constant $d_p>0$ is independent of $T,\,\lambda$ and $\mu$.
Since $p>2$, the conclusion follows by choosing $\rho>0$ sufficiently small.
\end{proof}

\begin{lemma}\label{lem 3.2}
Suppose that $2<p<4$ and $(V_{1})$--$(V_{3})$ hold. Then there exists
$\mu^{\ast}>0$ such that for each $T,\, \lambda > 0$ and $\mu \in (0,\,\mu^{\ast})$,
we have $I^T_{\lambda,\mu}(e_0)<0$ for some $e_0 \in C_{0}^{\infty}(\Omega)$ with $|\nabla e_0|_2>\rho$.
\end{lemma}

\begin{proof}
We first define the functional $J_{\lambda}:\: E_\lambda \rightarrow \mathbb{R}$ by
\begin{equation*}
  J_{\lambda} (u) = \frac{1}{2}\int_{\mathbb{R}^{3}} \left(|\nabla u|^{2}+\lambda
   V(x) u^{2}\right)dx - \frac{1}{p}\int_{\mathbb{R}^{3}}|u^+|^p\:dx.
\end{equation*}
Let $e \in C_{0}^{\infty}(\Omega)$ be a positive smooth function, then we have
$$
  J_{\lambda} (t e) = \frac{t^2}{2}\int_{\Omega} |\nabla e|^{2}\:dx
     - \frac{t^p}{p}\int_{\Omega}|e|^p\:dx \rightarrow -\infty
     \qquad \text{as $t\rightarrow \infty$}.
$$
Therefore, there exists $e_0 \in C_{0}^{\infty}(\Omega)$ with $|\nabla e_0|_2>\rho$
such that $J_{\lambda}(e_0) \leq -1$. Since
\begin{equation*}
  I^T_{\lambda,\mu}(e_0)= J_{\lambda}(e_0)+\frac{\mu}{4}\eta\left(\|e_0\|^2_{\lambda}/T^2\right)
      \int_{\mathbb{R}^{3}}\phi_{e_0}e_0^2\:dx
  \leq  -1 + \frac{\mu}{4}C_0|e_0|^4_{12/5},
\end{equation*}
there exists $\mu^\ast>0$ $($independent of $\lambda$ and $T$$)$ such that
$I^T_{\lambda,\mu}(e_0)<0$ for all $T,\,\lambda>0$ and $\mu \in (0,\,\mu^\ast)$.
The proof is thus finished.
\end{proof}

\begin{remark}\label{remark 3.1}
\rm
We would like to point out that, the function $e_0 \in C_{0}^{\infty}(\Omega)$
of Lemma \ref{lem 3.2} is positive, and it does not depend on $\,T,\,\lambda$ and $\mu$.
\end{remark}

We now consider the mountain pass value
\vspace{-0.1cm}
\begin{equation*}
  c^T_{\lambda,\mu}=\inf_{\gamma \in \Gamma}\max_{t \in [0,1]}
   I^T_{\lambda,\mu}(\gamma(t)),
\end{equation*}
where
$
  \Gamma=\left\{\gamma \in C\big([0,1],\, E_\lambda)\big):\:\gamma(0)=0,\:\gamma(1)=e_0\right\}.
$
From Lemmas \ref{lem 3.1}, \ref{lem 3.2} and Theorem \ref{th 2.2}, we thus deduce that
for each $T>0$, $\lambda\geq 1$ and $\mu \in (0,\,\mu^\ast)$, there exists a Cerami sequence
$\{u_{n}\}\subset E_\lambda$ (here we do not write the dependence on $T,\,\lambda$ and $\mu$)
such that
\begin{equation}\label{eq 3.2}
  I^T_{\lambda,\mu}(u_{n})\rightarrow c^T_{\lambda,\mu} \qquad \text{and} \qquad
  (1+\|u_n\|_\lambda)\left\|(I^T_{\lambda,\mu})'(u_{n})\right\|_{E'_\lambda}\rightarrow 0.
\end{equation}
Clearly, $c^T_{\lambda,\mu}\geq \alpha>0$. Next, we also provide an estimate on the upper bound
of $c^T_{\lambda,\mu}$ which is the key ingredient of the truncation technique.

\begin{lemma}\label{lem 3.3}
Suppose that $2<p<4$ and $(V_{1})$--$(V_{3})$ hold. Then for each
$T>0$, $\lambda \geq 1$ and $\mu \in (0,\,\mu^\ast)$, there exists $M>0$
$($independent of $\,T,\,\lambda$ and $\mu$$)$ such that $c^T_{\lambda,\mu}\leq M$.
\end{lemma}

\begin{proof}
By Remark \ref{remark 3.1}, it is easy to see that
\begin{equation*}
  I^T_{\lambda,\mu}(te_{0}) \leq \frac{t^2}{2}|\nabla e_0|_2^{2}
    +\frac{\mu^\ast}{4}C_0 t^4|e_0|^4_{12/5}
    - \frac{t^p}{p}|e_0|_p^p.
\end{equation*}
Consequently, there exists a constant $M>0$ $($independent of $T,\:\lambda$ and $\mu$$)$
such that
$$
  c^T_{\lambda,\mu}\leq \max_{t \in [0,1]}I^T_{\lambda,\mu}(te_{0})\leq M.
$$
This completes the proof.
\end{proof}

In the following key lemma, we shall show that for a given $T>0$ properly,
after passing to a subsequence,
the sequence $\{u_{n}\}$ given by \eqref{eq 3.2} satisfies $\|u_n\|_\lambda\leq T$,
and so $\{u_n\}$ is also a bounded Cerami sequence of $I_{\lambda,\mu}$
satisfying $\|u_n\|_\lambda\leq T$.

\begin{lemma}\label{lem 3.4}
Suppose that $2< p < 4$ and $(V_{1})$--$(V_{3})$ hold, and let $T=\sqrt{\frac{2p(M+1)}{p-2}}$.
Then there exists $\mu_\ast \in (0, \,\mu^\ast)$ such that,
for each $\lambda \geq 1$ and $\mu \in (0,\,\mu_{\ast})$,
if $\{u_{n}\}\subset E_\lambda$ is a sequence satisfying \eqref{eq 3.2},
then we have, up to a subsequence, $\|u_{n}\|_\lambda \leq T$.
In particular, this sequence $\{u_{n}\}$ is also a Cerami sequence
at level $c^T_{\lambda,\mu}$ for $I_{\lambda,\mu}$, i.e.,
\begin{equation*}
  I_{\lambda,\mu}(u_{n})\rightarrow c^T_{\lambda,\mu} \qquad \text{and} \qquad
  (1+\|u_n\|_\lambda)\left\|I'_{\lambda,\mu}(u_{n})\right\|_{E'_\lambda}\rightarrow 0,
\end{equation*}
\end{lemma}

\begin{proof}
We first show that $\|u_n\|_\lambda \leq \sqrt{2}T$ for $n$ large enough.
Suppose by contradiction that, there exists a subsequence of $\{u_{n}\}$,
still denoted by $\{u_{n}\}$, such that $\|u_n\|_\lambda > \sqrt{2}T$.
By \eqref{eq 3.1} and \eqref{eq 3.2}, we then obtain
\begin{align}\label{eq 3.3}
  c^T_{\lambda,\mu} &= \lim_{n\rightarrow\infty}\left(I^T_{\lambda,\mu}(u_{n})
     - \frac{1}{p} \left\langle (I^T_{\lambda,\mu})'(u_{n}),\:  u_{n}\right\rangle\right)\notag\\
  &= \lim_{n\rightarrow\infty}\Bigl(\bigl(\frac{1}{2}-\frac{1}{p}\bigr)\|u_n\|_\lambda^2 -
     \bigl(\frac{\mu}{p}-\frac{\mu}{4}\bigr)\eta\bigl(\|u_n\|^2_{\lambda}/T^2\bigr)
     \int_{\mathbb{R}^{3}}\phi_{u_n}u_n^2\:dx\notag\\
  &\ \quad  -\frac{\mu}{2p T^2}\eta'\left(\|u_n\|^2_{\lambda}/T^2\right)
      \|u_n\|_\lambda^2 \int_{\mathbb{R}^{3}}\phi_{u_n}u_n^2\:dx\Bigr)\notag\\
  &\geq 2(M+1),
\end{align}
which is a contradiction by Lemma \ref{lem 3.3}.

We may now complete the proof of the lemma. Suppose by contradiction that,
there exists no subsequence of $\{u_{n}\}$ which is uniformly bounded by $T$.
Then we deduce that $T< \|u_n\|_\lambda \leq \sqrt{2}T$ for $n$ large enough.
With a similar computation as \eqref{eq 3.3} and using the fact that $\eta$ is nonincreasing,
we conclude that
\begin{align*}
  c^T_{\lambda,\mu} &= \lim_{n\rightarrow\infty}\left(I^T_{\lambda,\mu}(u_{n})
     - \frac{1}{p} \left\langle (I^T_{\lambda,\mu})'(u_{n}),\:  u_{n}\right\rangle\right)\\
  &= \lim_{n\rightarrow\infty}\Bigl(\bigl(\frac{1}{2}-\frac{1}{p}\bigr)\|u_n\|_\lambda^2 -
     \bigl(\frac{\mu}{p}-\frac{\mu}{4}\bigr)\eta\bigl(\|u_n\|^2_{\lambda}/T^2\bigr)
     \int_{\mathbb{R}^{3}}\phi_{u_n}u_n^2\:dx\\
  &\ \quad  -\frac{\mu}{2p T^2}\eta'\left(\|u_n\|^2_{\lambda}/T^2\right)
      \|u_n\|_\lambda^2 \int_{\mathbb{R}^{3}}\phi_{u_n}u_n^2\:dx\Bigr)\\
  &\geq \liminf_{n\rightarrow\infty}\Bigl(\bigl(\frac{1}{2}-\frac{1}{p}\bigr)\|u_n\|_\lambda^2 -
     \bigl(\frac{\mu}{p}-\frac{\mu}{4}\bigr)C_0d_{12/5}^4\| u_n\|_\lambda^4 \Big)\\
  &\geq (M+1)-\frac{4p(4-p)}{(p-2)^2}\mu C_0d_{12/5}^4(M+1)^2,
\end{align*}
this is a contradiction by choosing $\mu_\ast>0$ sufficiently small. So the claim follows.
\end{proof}

We are now ready to give the compactness conditions for $I_{\lambda,\mu}$.
For this we need to establish the following parameter-dependent compactness lemma.

\begin{lemma}\label{lem 3.5}
Suppose that $2< p < 4$ and $(V_{1})$--$(V_{3})$ hold, and let $T=\sqrt{\frac{2p(M+1)}{p-2}}$.
Then there exists $\lambda^\ast > 1$ such that,
for each $\lambda \in (\lambda^\ast,\,\infty)$ and $\mu \in (0,\,\mu_{\ast})$,
if $\{u_{n}\}\subset E_\lambda$ is a sequence satisfying \eqref{eq 3.2},
then $\{u_{n}\}$ has a convergent subsequence in $E_\lambda$.
\end{lemma}

\begin{proof}
By Lemma \ref{lem 3.4}, we see that, up to a subsequence, $\|u_{n}\|_\lambda \leq T$.
Passing to a subsequence again if necessary, we may assume that there exists
$u \in E_\lambda$ such that
\begin{equation}\label{eq 3.4}
  u_n \rightharpoonup u \quad \text{in $E_{\lambda}$}\qquad \text{and}\qquad
  \lim_{n\rightarrow \infty } \int_{\mathbb{R}^{3}}\phi_{u_n}u_n^2\:dx \geq
  \int_{\mathbb{R}^{3}}\phi_{u}u^2\:dx
\end{equation}
Moreover, $u$ is a critical point of $I_{\lambda,\mu}$, and it follows that
\begin{equation}\label{eq 3.5}
   \langle I'_{\lambda,\mu}(u),\:u\rangle=\|u\|_\lambda^2  + \mu \int_{\mathbb{R}^{3}}\phi_{u}u^2\:dx
   - |u^+|_p^{p} = 0.
\end{equation}

Now we show that $u_n \rightarrow u$ in $E_\lambda$. Let $v_n:=u_n-u$.
It follows from $(V_2)$ that
\begin{equation*}
   |v_n|_2^2= \int_{\R^3 \setminus \mathcal{V}_b} v_n^2\:dx
    + \int_{\mathcal{V}_b} v_n^2 \:dx\leq \frac{1}{\lambda b}
    \|v_n\|^2_\lambda +o(1).
\end{equation*}
Then, by the H\"{o}lder and Sobolev inequalities, we have
\begin{equation*}
   |v_n|_p \leq |v_n|_2^\theta |v_n|_6^{1-\theta}
   \leq d_0|v_n|_2^\theta |\nabla v_n|_2^{1-\theta}
   \leq d_0(\lambda b)^{-\theta/2}\|v_n\|_\lambda +o(1),
\end{equation*}
where $\theta= \frac{6-p}{2p}$ and the constant $d_0>0$ is independent
of $b$ and $\lambda$.
Combining this with \eqref{eq 2.1}, \eqref{eq 3.4} and \eqref{eq 3.5}, we infer that
\begin{align*}
  o(1) &=  \bigl\langle I'_{\lambda,\mu}(u_n), \: u_n \bigr\rangle - \bigl\langle I'_{\lambda,\mu}(u), \: u \bigr\rangle \\
    & =  \|u_n\|_\lambda^2 + \mu \int_{\mathbb{R}^{3}}\phi_{u_n}u_n^2\:dx - |u_n^+|_p^{p}  - \|u\|_\lambda^2
           - \mu \int_{\mathbb{R}^{3}}\phi_{u}u^2\:dx + |u^+|_p^{p}\\
    & \geq \|v_n\|_\lambda^2 - |v_n^+|_p^{p} + o(1) \\
    & \geq  \|v_n\|_\lambda^2 - |v_n|_p^{p-2}|v_n|_p^{2}+o(1)\\
    & \geq \left[1- \left(2d_p T \right)^{p-2}d_0^2
     (\lambda b)^{-\theta} \right]\|v_n\|_\lambda^2+o(1).
\end{align*}
Hence, there exists $\lambda^\ast > 1$ such that $v_n \rightarrow 0$
in $E_\lambda$ for all $\lambda > \lambda^\ast$.
This completes the proof.
\end{proof}

\vskip 0.2 true cm

\begin{proof}[Proof of Theorem \ref{th 1.1}]
Let $T$ be defined as in Lemma \ref{lem 3.4}.
By Lemmas \ref{lem 3.1} and \ref{lem 3.2}, there exists $\mu^\ast>0$ such that
for every $\lambda\geq 1$ and $\mu \in (0,\,\mu^\ast)$, $I^T_{\lambda,\mu}$
possesses a Cerami sequence $\{u_{n}\}$ at the mountain pass level $c^T_{\lambda,\mu}$.
From Lemmas \ref{lem 3.3} and \ref{lem 3.4}, we thus deduce that there exists
$\mu_\ast \in (0,\,\mu^\ast)$  such that for every $\lambda \geq 1$ and
$\mu \in (0,\:\mu_\ast)$, after passing to a subsequence, $\{u_{n}\}$ is
a Cerami sequence of  $I_{\lambda,\mu}$ satisfying $\|u_n\|_\lambda\leq T$, i.e.,
\begin{equation*}
  \sup_{n \in \N} \|u_n\|_\lambda\leq T, \qquad I_{\lambda,\mu}(u_{n})
   \rightarrow c^T_{\lambda,\mu}  \qquad  \text{and}  \qquad
  (1+\|u_n\|_\lambda)\|I_{\lambda,\mu}'(u_{n})\|_{E'_\lambda}\rightarrow 0.
\end{equation*}
It follows from Lemma \ref{lem 3.5} that there exists $\lambda^\ast>1$ such that
for each $\lambda \in (\lambda^\ast,\,\infty)$ and $\mu \in (0,\,\mu_\ast)$,
the sequence $\{u_n\} $ has a convergent subsequence in $E_\lambda$.
We may then assume that $u_n \rightarrow u_{\lambda,\mu}$ as $n\rightarrow\infty$,
and thus
$$
  \|u_{\lambda,\mu}\|_\lambda\leq T, \qquad
  I_{\lambda,\mu}(u_{\lambda,\mu}) = c^T_{\lambda,\mu}
  \qquad \text{and} \qquad  I_{\lambda,\mu}'(u_{\lambda,\mu})=0.
$$
Consequently, from Lemma \ref{lem 2.1} we see that $u_{\lambda,\mu}$ is a positive solution
of \eqref{eq 1.2} for each $\lambda \in (\lambda^\ast,\,\infty)$ and  $\mu \in (0,\,\mu_\ast)$.
Moreover, since  $\big\langle I'_{\lambda,\mu}(u_{\lambda,\mu}), \:u_{\lambda,\mu}\big\rangle=0$
and $u_{\lambda,\mu}\neq 0$, we have
\begin{equation*}
  \|u_{\lambda,\mu}\|_\lambda^2 \leq |u_{\lambda,\mu}^+|_p^p \leq d_p^p\|u_{\lambda,\mu}\|_\lambda^p,
\end{equation*}
and hence there exists $\tau >0$ $($independent of $\lambda$ and $\mu$$)$
such that $ \|u_{\lambda,\mu}\|_\lambda \geq \tau$ for all $\lambda$ and $\mu$.
This ends the proof.
\end{proof}

\begin{proof}[Proof of Theorem \ref{th 1.2}]
The strategy of proof is inspired by \cite[Theorem 4.1]{Ruiz-2006}.
Suppose that $u \in E_{\lambda}$ is a nontrivial solution of \eqref{eq 1.2}. Multiplying
equation \eqref{eq 1.2} by $u$ and integrating by parts, we obtain
\begin{equation}\label{eq 3.6}
 \int_{\mathbb{R}^{3}}\left(|\nabla u|^{2}+\lambda V(x)u^{2}
 +\mu \phi_{u}u^{2}- |u|^p\right)\:dx = 0.
\end{equation}
By the definition of $\phi_u$, we get that
\begin{align*}
  \int_{\mathbb{R}^{3}} \phi_{u}u^{2} \: dx &= \int_{\mathbb{R}^{3}} \phi_{u} (-\triangle \phi_u)\:dx
  = \int_{\mathbb{R}^{3}} |\nabla \phi_{u}|^2\:dx,\\
   \int_{\mathbb{R}^{3}} |u|^{3} \: dx &= \int_{\mathbb{R}^{3}} (-\triangle \phi_u)|u|\:dx
  = \int_{\mathbb{R}^{3}} \nabla \phi_{u} \nabla |u|\:dx.
\end{align*}
This readily implies that
\begin{equation}\label{eq 3.7}
  \int_{\mathbb{R}^{3}}|u|^{3}\:dx = \int_{\mathbb{R}^{3}}\nabla \phi_{u} \nabla |u|\:dx
  \leq \frac{1}{4\mu}\int_{\mathbb{R}^{3}}|\nabla u|^{2}\:dx + \mu \int_{\mathbb{R}^{3}}\phi_{u}u^{2}\:dx.
\end{equation}
Inserting \eqref{eq 3.7}  into \eqref{eq 3.6}, we may then distinguish the following two cases:

\emph{Case 1: $p = 3$}. In this case, for $\lambda >0$ and $\mu > 1/4$ we have
\begin{align*}
 0&=  \int_{\mathbb{R}^{3}}\left(|\nabla u|^{2}+\lambda V(x)u^{2} +\mu \phi_{u}u^{2}- |u|^3\right)\:dx \\
 &\geq \bigl(1-\frac{1}{4\mu}\bigr)\int_{\mathbb{R}^{3}}|\nabla u|^{2}\:dx.
\end{align*}
This implies that $u$ must be equal to zero.

\emph{Case 2: $2< p < 3$}. In this case,  for $\lambda \geq 1/b$ and $\mu \geq 1 / \bigl[4(1-\left|\mathcal{V}_b\right|^{\frac{2}{3}}S^{-1})\bigr]$
we deduce that
\begin{align*}
 0&=  \int_{\mathbb{R}^{3}}\left(|\nabla u|^{2}+\lambda V(x)u^{2} +\mu \phi_{u}u^{2}- |u|^p\right)\:dx \\
 &\geq \bigl(1-\frac{1}{4\mu}\bigr)\int_{\mathbb{R}^{3}}|\nabla u|^{2}\:dx + \int_{\R^3 \setminus \mathcal{V}_b} u^{2}\:dx
 +\int_{\mathbb{R}^{3}}|u|^{3}\: dx-\int_{\mathbb{R}^{3}}|u|^p\:dx\\
  &\geq \bigl(1-\frac{1}{4\mu}-\left|\mathcal{V}_b\right|^{\frac{2}{3}}S^{-1}\bigr)\int_{\mathbb{R}^{3}}|\nabla u|^{2}\:dx + \int_{\R^3} u^{2}\:dx
 +\int_{\mathbb{R}^{3}}|u|^{3}\: dx-\int_{\mathbb{R}^{3}}|u|^p\:dx\\
  &\geq \int_{\mathbb{R}^{3}} \left(|\nabla u|^{2} + |u|^{3} -|u|^p \right)\:dx
\end{align*}
It is easy to check that, if $2< p < 3$, the function
$$
  h:\: [0,\:\infty) \rightarrow \R, \qquad h(t)= t^2+t^3-t^p
$$
is nonnegative and vanish only at zero. Hence, $u$ must be equal to zero.
The proof is complete.
\end{proof}

\section{Decay rate of positive solutions}\label{section 4}

\setcounter{equation}{0}\setcounter{section}{4}
$ $
\indent
In this section, we explore the decay rate of the positive solutions
for \eqref{eq 1.2} at infinity and give the proof of Theorem \ref{th 1.3}.
For this purpose, throughout this section we always assume that $2<p<4$ and the conditions $(V_{1})$,
$(V'_{2})$ and $(V_{3})$ hold, and assume moreover, for each $\lambda \in (\lambda^\ast,\,\infty)$
and $\mu \in (0,\,\mu_\ast)$, $u_{\lambda,\mu}$ is the positive solution
of \eqref{eq 1.2} obtained by Theorem \ref{th 1.1}.

At first, we give a crucial lemma in the study of the decay rate of solutions,
since it gives an important estimate involving the $L^\infty$-norm of solutions.
We sketch the proof by adopting some arguments which are related to
the Moser iterative method, see e.g. in \cite[p. 270]{Struwe1990}.

\begin{lemma}\label{lem 4.1}
The positive solutions $u_{\lambda,\mu}$ are in $L^\infty(\R^3) \cap  C_{loc}^{1,\alpha}(\R^3)$
for some $0<\alpha<1$. Moreover, there exists $C_0>0$ $($independent of $\lambda$ and $\mu$$)$
such that
\begin{equation*}
  |u_{\lambda,\mu}|_\infty \leq C_0 \qquad  \text{for all $\lambda$ and $\mu$}.
\end{equation*}
\end{lemma}
\begin{proof}
For each $m \in \N$ and $\beta>1$, we set
$$
  A_m = \{x\in \R^3:\: u_{\lambda,\mu}^{\beta-1}(x) \leq m\}, \quad
  B_m = \{x\in \R^3:\: u_{\lambda,\mu}^{\beta-1}(x) > m\}
$$
and
\begin{equation*}
  v_m =
  \left\{
    \begin{array}{ll}
      u_{\lambda,\mu}^{2\beta-1}   \qquad & \hbox{in  $A_m$,} \vspace{0.1cm}\\
      m^2 u_{\lambda,\mu} & \hbox{in  $B_m$.}
    \end{array}
  \right.
\end{equation*}
A direct computation yields that $v_m \in E_\lambda$, $v_m \leq u_{\lambda,\mu}^{2\beta-1}$
and
\begin{equation}\label{eq 4.1}
  \nabla v_m =
  \left\{
    \begin{array}{ll}
     (2\beta -1) u_{\lambda,\mu}^{2\beta-2}\nabla u_{\lambda,\mu} \ \quad & \hbox{in $A_m$,} \vspace{0.1cm}\\
      m^2\nabla u_{\lambda,\mu}    & \hbox{in  $B_m$.}
    \end{array}
  \right.
\end{equation}
Testing \eqref{eq 1.2} with $v_m$, we obtain
\begin{equation}\label{eq 4.2}
  \int_{\R^3} (\nabla u_{\lambda,\mu}\nabla v_m + \lambda V(x) u_{\lambda,\mu} v_m )\:dx
    + \mu\int_{\mathbb{R}^{3}}\phi_{u_{\lambda,\mu}}u_{\lambda,\mu}v_m \:dx
   = \int_{\R^3} u^{p-1}_{\lambda,\mu} v_m \:dx.
\end{equation}
From \eqref{eq 4.1}, we can easily see that
\begin{equation}\label{eq 4.3}
  \int_{\R^3} \nabla u_{\lambda,\mu}\nabla v_m \:dx =
  (2\beta-1)\int_{A_m} u_{\lambda,\mu}^{2\beta-2} |\nabla u_{\lambda,\mu}|^2\:dx
  + m^2 \int_{B_m} |\nabla u_{\lambda,\mu}|^2\:dx.
\end{equation}
Let
\begin{equation*}
  w_m =
  \left\{
    \begin{array}{ll}
      u_{\lambda,\mu}^{\beta}  \qquad & \hbox{in $A_m$,} \vspace{0.1cm}\\
      m u_{\lambda,\mu} & \hbox{in $B_m$,}
    \end{array}
  \right.
\end{equation*}
then we have  $w_m \in E_\lambda$, $w_m \leq u_{\lambda,\mu}^{\beta}$ and
\begin{equation*}
  \nabla w_m =
  \left\{
    \begin{array}{ll}
     \beta u_{\lambda,\mu}^{\beta-1}\nabla u_{\lambda,\mu} \quad & \hbox{in \ $A_m$,} \vspace{0.1cm}\\
      m \nabla u_{\lambda,\mu}    & \hbox{in \ $B_m$},
    \end{array}
  \right.
\end{equation*}
which implies that
\begin{equation*}
  \int_{\R^3}|\nabla w_m|^2 \:dx =
  \beta^2 \int_{A_m} u_{\lambda,\mu}^{2\beta-2} |\nabla u_{\lambda,\mu}|^2\:dx
  + m^2 \int_{B_m} |\nabla u_{\lambda,\mu}|^2\:dx.
\end{equation*}
Combining this with \eqref{eq 4.2} and \eqref{eq 4.3}, we infer that
\begin{equation*}
  \int_{\R^3} |\nabla w_m|^2 \:dx \leq \left[\frac{(\beta-1)^2}{2\beta -1}+1\right]
    \int_{\R^3} \nabla u_{\lambda,\mu}\nabla v_m \:dx
   \leq \beta^2 \int_{\R^3} u^{p-2}_{\lambda,\mu} w_m^2 \:dx.
\end{equation*}
It then follows from the Sobolev and the H\"{o}lder inequalities and \eqref{eq 1.4} that
\begin{align*}
   \left(\int_{A_m} u_{\lambda,\mu}^{6\beta}\: dx \right)^{1/3}
   & = \left(\int_{A_m} |w_m|^6 \:dx  \right)^{1/3}
     \leq S^{-1}  \int_{\R^3} |\nabla w_m|^2 \:dx \\
     &\leq S^{-1} \beta^2 \int_{\R^3} u^{p-2}_{\lambda,\mu} w_m^2 \:dx
       \leq S^{-1} \beta^2  |u_{\lambda,\mu}|_p^{p-2} | w_m|_p^2\\
     &\leq  S^{-1} \beta^2 (d_pT)^{p-2}|u_{\lambda,\mu}|_{p\beta}^{2\beta}.
\end{align*}
Hence, we may let $m\rightarrow \infty$ to derive that
\begin{equation}\label{eq 4.4}
  |u_{\lambda,\mu}|_{6\beta} \leq \beta^{\frac{1}{\beta}}\left(S^{-1}
   (d_pT)^{p-2}\right)^{\frac{1}{2\beta}}|u_{\lambda,\mu}|_{p\beta}.
\end{equation}
Set $\sigma =6 / p$, then we see that $\sigma>1$.
When $\beta = \sigma$ in \eqref{eq 4.4}, we yield that
\begin{equation*}
  |u_{\lambda,\mu}|_{6\sigma}  \leq \sigma^{\frac{1}{\sigma}}\left(S^{-1}
   (d_pT)^{p-2}\right)^{\frac{1}{2\sigma}}|u_{\lambda,\mu}|_{6}.
\end{equation*}
Arguing by iteration, let $\beta=\sigma^j$ in \eqref{eq 4.4},
we may show that
\begin{align}\label{eq 4.5}
 |u_{\lambda,\mu}|_{6\sigma^j}  &\leq \sigma^{\frac{1}{\sigma}+\frac{2}{\sigma^2}+\cdots+
   \frac{j}{\sigma^j}}\left(S^{-1}(d_pT)^{p-2}\right)^{\frac{1}{2}\left(\frac{1}{\sigma}
   +\frac{1}{\sigma^2}+\cdots+\frac{1}{\sigma^j}\right)}|u_{\lambda,\mu}|_{6}\notag\\
  &\leq \sigma^{\frac{\sigma}{(\sigma-1)^2}}\left(S^{-1}(d_pT)^{p-2}\right)^{\frac{1}{2(\sigma-1)}}|u_{\lambda,\mu}|_{6}.
\end{align}
Let $j\rightarrow \infty$ in \eqref{eq 4.5}, we may use \eqref{eq 1.4} again to obtain
\begin{align*}
  |u_{\lambda,\mu}|_{\infty} &\leq \sigma^{\frac{\sigma}{(\sigma-1)^2}}\left(S^{-1}(d_pT)^{p-2}\right)^{\frac{1}{2(\sigma-1)}}|u_{\lambda,\mu}|_{6}\\
  & \leq \sigma^{\frac{\sigma}{(\sigma-1)^2}}\left(S^{-1}(d_pT)^{p-2}\right)^{\frac{1}{2(\sigma-1)}}
  \bigl(S^{-\frac{1}{2}}T\bigr).
\end{align*}
Let
\begin{equation*}
  C_0=\sigma^{\frac{\sigma}{(\sigma-1)^2}}\left(S^{-1}(d_pT)^{p-2}\right)^{\frac{1}{2(\sigma-1)}}
  \bigl(S^{-\frac{1}{2}}T\bigr)\qquad \text{with $\sigma=6/p$},
\end{equation*}
then we get
\begin{equation*}
  |u_{\lambda,\mu}|_\infty \leq C_0 \qquad  \text{for all $\lambda$ and $\mu$}.
\end{equation*}
Thus, in weak sense we have
\begin{equation*}
  -\triangle u_{\lambda,\mu} = u_{\lambda,\mu}^{p-1}- \bigl(\lambda V(x) + \mu \phi_{u_{\lambda,\mu}}(x)\bigr)u_{\lambda,\mu}
  \in L_{loc}^{q} (\R^3) \qquad \text{for all $q\geq 1$}.
\end{equation*}
It then follows from \cite[Theorem 9.11]{Gilbarg1983} that $u_{\lambda,\mu} \in W_{loc}^{2,q}$ for all $q\geq 1$, whence also
$u_{\lambda,\mu} \in  C_{loc}^{1,\alpha}(\R^3)$ for some $0<\alpha<1$ by the Sobolev embedding theorem.
The proof is thus finished.
\end{proof}

\begin{remark}
\rm
We note that, the condition $(V'_2)$ was not used in Lemma \ref{lem 4.1}.
Consequently, under the assumptions of Theorem \ref{th 1.1},
Lemma \ref{lem 4.1} still holds.
\end{remark}

We need the following lemma, which is a simple consequence of Lemma \ref{lem 4.1} and \eqref{eq 1.4}.

\begin{lemma}\label{lem 4.2}
There exist $A,\,R>0$ $($independent of $\lambda$ and $\mu$$)$ such that
\begin{equation*}
  u_{\lambda,\mu}(x) \leq A \lambda^{- 1/2}\qquad
  \text{for all $\lambda,\:\mu$ and $|x| > R$}.
\end{equation*}
\end{lemma}
\begin{proof}
Let $c(x)=- u_{\lambda,\mu}^{p-2}(x)$ for $x \in \R^3$, we deduce from
Lemma \ref{lem 4.1} that
\begin{equation*}
  |\,c(x)| \leq C_0^{p-2} \qquad \text{for all $x \in \R^3$.}
\end{equation*}
Observe that
\begin{equation*}
  -\triangle u_{\lambda,\mu} + c(x) u_{\lambda,\mu} \leq 0,
\end{equation*}
it then follows from \cite [Theorem 8.17]{Gilbarg1983} that
there exists $C_1>0$ $($independent of $\lambda$ and $\mu$$)$ such that
\begin{equation}\label{eq 4.6}
  \sup_{x \in B_1(y)} u_{\lambda,\mu}(x) \leq C_1 |u_{\lambda,\mu}|_{L^2(B_2(y))}
   \qquad \text{for all $y \in \R^3$.}
\end{equation}
By $(V'_2)$, there exists $R_1>0$ such that $\mathcal{V}_b \subset B_{R_1}(0)$,
and so
\begin{equation}\label{eq 4.7}
  V(x)\geq b \qquad  \text{for \ $|x|\geq R_1$.}
\end{equation}
Combining this with \eqref{eq 1.4} gives
\begin{equation}\label{eq 4.8}
  |u_{\lambda,\mu}|_{L^2(B_2(y))} \leq  T(b\lambda )^{-1/2} \qquad
  \text{for all $|y|\geq R_1+2$}.
\end{equation}
Now \eqref{eq 4.6} and \eqref{eq 4.8} yield that
\begin{equation*}
  u_{\lambda,\mu}(x) \leq C_1 T ( b\lambda )^{-1/2} \qquad
  \text{for all $\lambda,\:\mu$ and $|x|> R_1 +1$},
\end{equation*}
and thus the claim follows by choosing $A=C_1 T b^{-1/2}$ and $R=R_1+1$.
\end{proof}

\begin{proof}[Proof of Theorem \ref{th 1.3}]
Set
$$
  W_{\lambda,\mu} (x)= \lambda V(x) - |u_{\lambda,\mu}(x)|^{p-2}
  \qquad \text{for all  $x \in \R^3$.}
$$
By \eqref{eq 4.7} and Lemma \ref{lem 4.1}, there exists $\Lambda^\ast> \lambda^\ast$
such that for each $\lambda \in (\Lambda^\ast,\,\infty)$ and $\mu \in (0,\,\mu_\ast)$,
we have
$$
  W_{\lambda,\mu} (x)\geq \frac{b}{2}\lambda := \beta^2 \lambda \qquad
  \text{for all  $|x| > R$},
$$
and hence
\begin{equation}\label{eq 4.9}
  -\triangle u_{\lambda,\mu}(x)+ \beta^2 \lambda  u_{\lambda,\mu}(x) \leq 0 \qquad \text{for $|x|> R$}.
\end{equation}
Fix $\varphi_\lambda(x)= A \lambda^{-1/2} \exp\bigl(- \beta \lambda^{1/2}(|x|-R)\bigr)$,
we may deduce from Lemma \ref{lem 4.2} that
\begin{equation}\label{eq 4.10}
  \varphi_\lambda(x) \geq u_{\lambda,\mu}(x)\qquad \text{for  all $|x|= R$.}
\end{equation}
It is easy to see that
\begin{equation}\label{eq 4.11}
  \triangle \varphi_\lambda (x) \leq  \beta^2 \lambda \varphi_\lambda (x)  \qquad \text{for  all $|x|\neq 0$.}
\end{equation}
Define $\psi_\lambda=\varphi_\lambda - u_{\lambda,\mu}$. Using \eqref{eq 4.9}, \eqref{eq 4.10}
and \eqref{eq 4.11}, we yield that
\begin{equation*}
\left\{
  \begin{array}{ll}
    - \Delta \psi_\lambda (x) + \beta^2 \lambda \psi_\lambda (x) \geq 0 \quad & \hbox{in $\,|x|> R$,}
        \vspace{0.1cm}\\
    \psi_\lambda(x) \geq 0, & \hbox{on $|x|= R$.}
  \end{array}
\right.
\end{equation*}
The maximum principle (see e.g. \cite[Theorem 8.1]{Gilbarg1983}) implies that $\psi_\lambda (x)\geq 0$
for all $|x|\geq R$, and thus the claim follows.
\end{proof}

\section{Asymptotic behavior of positive solutions}\label{section 5}

\setcounter{equation}{0}\setcounter{section}{5}
$ $
\indent
In this section, we investigate the asymptotic behavior of positive solutions
for \eqref{eq 1.2} and give the proofs of Theorems \ref{th 1.4}, \ref{th 1.5}
and \ref{th 1.6}.

\begin{proof}[Proof of Theorem \ref{th 1.4}]
We follow the argument in \cite{Bartsch2001} (or see \cite{Dingyanheng2007,Zhaoleiga2013}).
Let $\mu \in (0,\,\mu_{\ast})$ be fixed, then for any sequence
$\lambda_{n}\rightarrow \infty$, let $u_{n}:=u_{\lambda_{n},\mu}$
be the positive solution of \eqref{eq 1.2} obtained by Theorem \ref{th 1.1}.
It follows from \eqref{eq 1.4} that
\begin{equation}\label{eq 5.1}
  0 < \tau \leq \|u_{n}\|_{\lambda_{n}} \leq T \qquad
  \text{for all $n$}.
\end{equation}
Thus, up to a subsequence, we may assume that
\begin{equation}\label{eq 5.2}
 \left\{
   \begin{array}{ll}
     u_{n} \rightharpoonup u_{\mu}  & \hbox{in \ $E$,} \vspace{0.1cm} \\
     u_{n} \rightarrow u_{\mu} & \hbox{in \ $L_{loc}^{s}(\mathbb{R}^{3})$
          \ \text{for} \ $s \in [2,\,6)$,}  \vspace{0.1cm}\\
     u_{n} \rightarrow  u_{\mu}  & \hbox{a.e. on $\R^3$.}
   \end{array}
 \right.
\end{equation}
By \eqref{eq 5.1}, \eqref{eq 5.2} and Fatou's lemma, we have
\begin{equation*}
  \int_{\mathbb{R}^{3}} V(x)u_{\mu}^{2} \: dx \leq \liminf_{n\rightarrow\infty}\int_{\mathbb{R}^{3}}
   V(x)u_{n}^{2} \:dx \leq \liminf_{n\rightarrow\infty}\frac{\|u_{n}\|_{\lambda_{n}}^{2}}{\lambda_{n}}=0.
\end{equation*}
Hence, $u_{\mu}=0$ a.e. in $\mathbb{R}^{3}\setminus V^{-1}(0)$,
and so $u_{\mu} \in H_{0}^{1}(\Omega)$ by the condition $(V_{3})$.

Now we show that $u_{n}\rightarrow u_{\mu}$ in $L^{s}(\mathbb{R}^{3})$
for $2<s<6$. Otherwise, by Lions' vanishing lemma (see e.g. \cite{LionsPL1984,Willem1996})
there exist $\delta,\,r>0$ and $x_{n}\in \mathbb{R}^{3}$ such that
\begin{equation*}
  \int_{B_{r}(x_{n})} (u_{n}-u_{\mu})^{2}\: dx \geq \delta.
\end{equation*}
This implies that $|x_{n}|\rightarrow\infty$, and so $\bigl|B_{r}(x_{n})\cap
\{x\in \mathbb{R}^{3}: \: V(x)<b\}\bigr|\rightarrow0$.
By the H\"{o}lder inequality, we then conclude that
\begin{equation*}
  \int_{B_{r}(x_{n})\cap\{V<b\}} (u_{n}-u_{\mu})^{2}\: dx \rightarrow 0.
\end{equation*}
Consequently, we get
\begin{equation*}
  \begin{split}
    \|u_{n}\|_{\lambda_{n}}^{2}&\geq \lambda_{n} b \int_{B_{r}(x_{n})\cap\{V\geq b\}}
      u_{n}^{2}\:dx= \lambda_{n} b \int_{B_{r}(x_{n})\cap\{V \geq b\}} (u_{n}-u_{\mu})^{2}\:dx\\
    &= \lambda_{n}b\bigg(\int_{B_{r}(x_{n})}(u_{n}-u_{\mu})^{2} \:dx
       - \int_{B_{r}(x_{n})\cap\{V< b\}}(u_{n}-u_{\mu})^{2}\:dx\bigg)\\
    &\rightarrow\infty,
\end{split}
\end{equation*}
which contradicts \eqref{eq 5.1}.

We then prove that $u_{n}\rightarrow u_{\mu}$  in $E$.
Since
$$
 \big\langle I'_{\lambda_{n},\mu}(u_{n}), \: u_{n}\big\rangle
 =\big\langle I'_{\lambda_{n},\mu}(u_{n}), \: u_{\mu}\big\rangle=0,
$$
we have
\begin{equation}\label{eq 5.3}
  \|u_{n}\|_{\lambda_{n}}^{2} + \mu\int_{\mathbb{R}^{3}}\phi_{u_n}u_n^2\:dx
   = |u_{n}^+|_p^p,
\end{equation}
\begin{equation}\label{eq 5.4}
  \|u_{\mu}\|^{2} + \mu\int_{\mathbb{R}^{3}}\phi_{u_\mu}u_\mu^2\:dx
   =|u_{\mu}^+|_p^p+o(1).
\end{equation}
By \eqref{eq 5.2} and Fatou's Lemma, after passing to subsequence, we yield that
\begin{equation}\label{eq 5.5}
  \lim_{n\rightarrow\infty} \int_{\mathbb{R}^{3}}\phi_{u_n}u_n^2\:dx
  \geq \int_{\mathbb{R}^{3}}\phi_{u_\mu}u_\mu^2\:dx.
\end{equation}
From \eqref{eq 5.3}--\eqref{eq 5.5}, we thus deduce that
$$
  \lim\limits_{n\rightarrow\infty}\|u_{n}\|_{\lambda_{n}}^{2} \leq \|u_{\mu}\|^{2}.
$$
It then follows from the weakly lower semi-continuity of norm that
\begin{equation}\label{eq 5.6}
  \|u_{\mu}\|^{2} \leq \liminf\limits_{n\rightarrow\infty}\|u_{n}\|^{2}
   \leq  \limsup\limits_{n\rightarrow\infty}\|u_{n}\|^{2}
  \leq \lim\limits_{n\rightarrow\infty}\|u_{n}\|_{\lambda_{n}}^{2}
  \leq \|u_{\mu}\|^{2},
\end{equation}
Consequently, we yield that $u_{n}\rightarrow u_{\mu}$ in E.

Finally, we only need to show that $u_{\mu}$ is a
positive solution of $(\mathcal{P}_{\infty, \mu})$.
Now for any $v \in C_{0}^{\infty}(\Omega)$, since
$\big\langle I'_{\lambda_{n},\mu}(u_{n}), \: v\big\rangle=0$,
it is easy to check that
\begin{equation*}
  \int_{\mathbb{R}^{3}}\nabla u_{\mu}\nabla v\:dx
  +\mu \int_{\mathbb{R}^{3}}\phi_{u_\mu}u_\mu v\:dx
  =\int_{\mathbb{R}^{3}}|u_{\mu}^+|^{p-2}u_{\mu}^+ v\:dx,
\end{equation*}
i.e., $u_{\mu}$ is a nonnegative solution of $(\mathcal{P}_{\infty, \mu})$ by
the density of $C_{0}^{\infty}(\Omega)$ in $H_{0}^{1}(\Omega)$.
By \eqref{eq 5.1} and \eqref{eq 5.6}, we infer that
$$
  \|u_\mu\|=\lim\limits_{n\rightarrow\infty}\|u_{n}\|_{\lambda_{n}}\geq \tau >0,
$$
and so $u_\mu \neq 0$. Therefore, the strong maximum principle
implies that $u_\mu > 0$ in $\R^3$. The proof is thus finished.
\end{proof}

\begin{proof}[Proof of Theorem \ref{th 1.5}]
Let $\lambda \in (\lambda^\ast,\,\infty)$ be fixed, then for any sequence $\mu_{n}\rightarrow 0$,
let $u_{n}:=u_{\lambda,\mu_n}$ be the positive solution of $(\mathcal{P}_{\lambda,\mu_n})$
obtained by Theorem \ref{th 1.1}. It follows from \eqref{eq 1.4} that
\begin{equation}\label{eq 5.7}
  0< \tau < \|u_{n}\|_{\lambda}\leq T \qquad
  \text{for all $n$}.
\end{equation}
Passing to a subsequence if necessary, we may assume that
$ u_{n} \rightharpoonup u_{\lambda}$ in $E_\lambda$.
Note that $I'_{\lambda,\mu_n}(u_{n}) = 0$,
we may deduce that $ u_{n} \rightarrow u_{\lambda}$ in $E_\lambda$
as the proof of Lemma \ref{lem 3.5}.

To complete the proof, it suffices to show that $u_{\lambda}$
is a positive solution of $(\mathcal{P}_{\lambda,0})$.
Now for any $v \in E_\lambda$, since
$\big\langle I'_{\lambda}(u_{n}),\:v\big\rangle=0$,
it is easy to check that
\begin{equation*}
  \int_{\mathbb{R}^{3}} (\nabla u_{\lambda}\nabla v + \lambda V(x)u_{\lambda} v)\:dx
  =\int_{\mathbb{R}^{3}}|u_{\lambda}^+|^{p-2}u_{\lambda}^+ v\:dx,
\end{equation*}
i.e., $u_{\lambda}$ is a nonnegative solution of $(\mathcal{P}_{\lambda,0})$.
Then, by \eqref{eq 5.7} we see that $u_\lambda \neq 0$. Therefore,
the strong maximum principle implies that $u_\lambda > 0$ in $\R^3$.
This completes the proof.
\end{proof}

\begin{proof}[Proof of Theorem \ref{th 1.6}]
The proof of  Theorem \ref{th 1.6} is similar to the proof of
Theorem \ref{th 1.4}, and we leave the detail to the reader.
\end{proof}

\noindent\textbf{Acknowledgements}
\vskip 0.2 true cm

The authors would like to express sincere thanks to the anonymous referee for
his/her carefully reading the manuscript and  valuable comments
and  suggestions.


\end{document}